\documentclass{amsart}

\usepackage[dvips]{graphicx}

\usepackage{ccfonts}
\usepackage[T1]{fontenc}
\usepackage[a4paper,top=46mm,bottom=46mm,left=37mm,right=37mm]{geometry}

\newtheorem{theorem}{Theorem}[section]

\newtheorem{lemma}{Lemma}[section]

\newtheorem{cor}{Corollary}[section]

\numberwithin{equation}{section}

\title{\uppercase{On an Inverse Problem for Sturm-Liouville Equation}}

\author[D. Karahan and Kh. R. Mamedov ]{ D\"one Karahan and Khanlar R. Mamedov}
\address{Department of Mathematics\newline \indent
Harran University\newline \indent
Sanlıurfa, 63000
Turkey}
\email{dkarahan@harran.edu.tr}

\address{
Department of Mathematics
 \newline \indent
Mersin University\newline \indent
 Mersin, 33000, Turkey}
\email{hanlar@mersin.edu.tr}



\date{}
\begin{document}

\begin{abstract}
In this study, the theorem on necessary and sufficient conditions for the solvability of inverse problem for Sturm-Liouville operator with discontinuous coefficient is proved and the algorithm of reconstruction of potential from spectral data (eigenvalues and normalizing numbers) is given.
\end{abstract}

\maketitle

\section{Introduction}
We consider the boundary value problem 
\begin{equation}
-y''+q(x)y=\lambda ^{2}\rho (x)y,\ 0\leq x\leq \pi,  \label{1}
\end{equation}
\begin{equation}
y'(0)=0, \quad y(\pi )=0,  \label{2}
\end{equation}
where $q\left( x\right) \in L_{2}\left( 0,\pi \right) $ is a real-valued function, $\rho(x)$ is a piecewise continuous function, $\lambda $ is a complex parameter. This spectral problem appears while solving wave or heat equations for nonhomogeneous density of the material \cite{Hal}, \cite{Tik-Sam}. Physical applications of discontinuous Sturm-Liouville problem are given in \cite{Tikh}-\cite{Fre}.

For simplicity, we will assume that the density function has only one discontinuity point such that   
\begin{equation}
\rho (x)=\left\{ 
\begin{array}{c}
1,\ \ 0\leq x\leq a, \\ 
\alpha ^{2},\ a<x\leq \pi, 
\end{array}%
\right.   \label{3}
\end{equation}
where $ 0<\alpha\neq 1$.

Direct problem of spectral analysis for Sturm-Liouville problem is investigated properties of eigenvalues and eigenfunctions, finding normalizing numbers, spectrum set of the boundary value problem, scattering data and some other values. It is important to investigate these properties. Inverse problem of spectral analysis is to final the coefficient of the equation for given spectral data. This has to be done uniquely, so that it gives the uniqueness of the inverse problem. In the process of the solution of the inverse problem giving an algorithm for constructing the potential is important. For $\rho(x)\equiv 1$, solutions of inverse problem for equation (\ref{1}) is given by \cite{Lev-Gas}-\cite{Gul}. For $\rho(x)\neq 1$, under different boundary conditions similar problem is solved in \cite{Akh1}-\cite{Mam-Kar2}. When boundary conditions contain spectral parameter, it is solved by \cite{Mam-Cet1}, \cite{Mam-Cet2}.

The inverse problem for this equation is to find necessary and sufficient conditions for any data set to be spectral data. The main of this work is to find these conditions for (\ref{1}), (\ref{2}) boundary value problem. Firstly spectral data is defined. Characteristic properties of these values are investigated in \cite{Mam-Kar} and also uniqueness of the solution of the inverse problem is proved.

Consequently, in this work for (\ref{1}), (\ref{2}) spectral problem, solution of the inverse problem is given with respect to the spectral data.

For (\ref{1}), (\ref{2}) boundary value problem in \cite{Mam-Kar}, it is shown that the real numbers $\left\{ \lambda _{n}^{2},\alpha _{n}\right\}_{n\geq 1}$ satisfy the following   
\begin{equation}
\lambda _{n}=\lambda _{n}^{0}+\frac{d_{n}}{\lambda _{n}^{0}}+\frac{k_{n}}{n}, \quad \alpha_n=\alpha _{n}^{0}+\frac{t_{n}}{n}, \quad \{k_n\},\{t_n\}\in l_2, \label{4} 
\end{equation}
where $\lambda _{n}^{0}$ are zeros of the function 
\begin{equation*}
\Delta _{0}(\lambda )=\frac{1}{2}(1+\frac{1}{\alpha})\cos \lambda \mu ^{+}(\pi )+
\frac{1}{2}(1-\frac{1}{\alpha})\cos \lambda \mu ^{-}(\pi ),
\end{equation*}
\begin{equation*}
d_{n}=\frac{h^{+}\sin \lambda _{n}^{0}\mu ^{+}(\pi )+h^{-}\sin \lambda
_{n}^{0}\mu ^{-}(\pi )}{\frac{1}{2}(1+\frac{1}{\alpha})\mu ^{+}(\pi )\sin \lambda
_{n}^{0}\mu ^{+}(\pi )+\frac{1}{2}(1-\frac{1}{\alpha})\mu ^{-}(\pi )\sin \lambda
_{n}^{0}\mu ^{-}(\pi )}
\end{equation*}
is a bounded sequence.

In \cite{Akh2} it is proved, that the solution $\varphi (x,\lambda )$ of the
equation (\ref{1}) with initial date $\varphi (0,\lambda )=1,$ $\varphi'(0,\lambda )=0$ can be represented as
\begin{equation}
\varphi (x,\lambda )=\varphi _{0}(x,\lambda )+\int_{0}^{\mu
^{+}(x)}A(x,t)\cos \lambda tdt,  \label{5}
\end{equation}%
where $A(x,t)$ belongs to the space $L_{2}(0,\pi )$ for each fixed $x\in
[0,\pi ]$ and is related to the coefficient $q(x)$ of the equation (\ref{1}) by the formula:
\begin{equation}
\frac{d}{dx}A(x,\mu ^{+}(x))=\frac{1}{4\sqrt{\rho (x)}}\left( 1+\frac{1}{%
\sqrt{\rho (x)}}\right) q(x),  
\end{equation}%
\begin{equation}
\varphi _{0}(x,\lambda )=\frac{1}{2}\left( 1+\frac{1}{\sqrt{\rho (x)}}%
\right)\cos \lambda \mu ^{+}(x)+\frac{1}{2}\left( 1-\frac{1%
}{\sqrt{\rho (x)}}\right)\cos \lambda \mu ^{-}(x)
\end{equation}%
is the solution of (\ref{1}) when $q(x)\equiv 0,$%
\begin{equation}
\mu ^{+}(x)=\pm x\sqrt{\rho (x)}+a\left( 1\mp \sqrt{\rho (x)}\right). 
\end{equation}  

The characteristic function $\Delta(\lambda)$ of the problem (\ref{1}), (\ref{2}) is
\begin{equation*}
\Delta(\lambda):=<\varphi(x,\lambda),\psi(x,\lambda)>=\varphi(x,\lambda)\psi'(x,\lambda)-\varphi'(x,\lambda)\psi(x,\lambda)
\end{equation*} 
where $\Delta(\lambda)$ is independent from $x\in[0,\pi]$. Substituting $x=0$ and $x=\pi$ into above the equation, we get 
\begin{equation*}
\Delta(\lambda)=\varphi(\pi,\lambda)=\psi'(0,\lambda).
\end{equation*}

\begin{theorem}
For each fixed $x\in \lbrack 0,\pi ]$ the kernel $A(x,t)$ from the
representation (\ref{5})satisfies the following linear functional integral
equation 
\begin{equation}
\frac{2}{1+\sqrt{\rho (t)}}A\left( x,\mu ^{+}(t)\right) +\frac{1-\sqrt{\rho
(2a-t)}}{1+\sqrt{\rho (2a-t)}}A\left( x,2a-t\right) +  \notag
\end{equation}%
\begin{equation}
+F(x,t)+\int_{0}^{\mu ^{+}(x)}A(x,\xi )F_{0}(\xi ,t)d\xi =0,\qquad 0<t<x \label{6}
\end{equation}%
where%
\begin{equation}
F_{0}(x,t)=\sum_{n=1}^{\infty }\left( \frac{\varphi _{0}(t,\lambda _{n})\cos
\lambda _{n}x}{\alpha _{n}}-\frac{\varphi _{0}(t,\lambda _{n}^{0})\cos
\lambda _{n}^{0}x}{\alpha _{n}^{0}}\right)  \label{7}
\end{equation}%
\begin{equation}
F(x,t)=\frac{1}{2}\left( 1+\frac{1}{\sqrt{\rho (x)}}\right) F_{0}(\mu
^{+}(x),t)+\frac{1}{2}\left( 1-\frac{1}{\sqrt{\rho (x)}}\right) F_{0}(\mu
^{-}(x),t)  \label{8}
\end{equation}%
$\left\{ \lambda _{n}^{0}\right\} ^{2}$ are eigenvalues and $\alpha _{n}^{0}$
are norming constants of the boundary value problem (\ref{1}), (\ref{2})
when $q(x)\equiv 0.$
\end{theorem}

\begin{theorem}
For each fixed $x\in \lbrack 0,\pi ]$ main equation (\ref{6}) has a unique
solution $A(x,.)\in L_{2,\rho }\left( 0,\mu ^{+}(x)\right) $.
\end{theorem}
The proof of Theorem 1.1 and Theorem 1.2 is given in \cite{Mam-Kar2}.

\section{Sufficient conditions for solvability of the inverse problem}

Assume that the real numbers $\left\{ \lambda _{n}^{2},\alpha _{n}\right\}_{n\geq 1}$ is given by the formula (\ref{4}). 
Now, let's construct $ F_0(x,t)$ and $F(x,t)$ functions by using the formulas (\ref{7}), (\ref{8}) and write the integral equation (\ref{6}).

We determine $A(x,t)$ from the main equation (\ref{6}). We shall construct the function $\varphi(x,\lambda)$ with the formula (\ref{5}) i.e.
 \begin{equation*}
\varphi (x,\lambda ):=\varphi _{0}(x,\lambda )+\int_{0}^{\mu
^{+}(x)}A(x,t)\cos \lambda tdt,
\end{equation*}
and the function $q(x)$ with formula 
\begin{equation}
q(x):=\frac{4\rho (x)}{\sqrt{\rho (x)}+1}\frac{d}{dx}A\left( x,\mu
^{+}(x)\right) .  \label{13}
\end{equation}
Denote 
\[
b(x):=\sum_{n=1}^{\infty }\left( \frac{\cos \lambda _{n}x}{\alpha _{n}\lambda^2
_{n}}-\frac{\cos \lambda _{n}^{0}x}{\alpha _{n}^{0}{\lambda _{n}^{0}}^{2}}\right) .
\]%
Similar to Lemma 1.3.4 in \cite{Yur}, it is shown that $b(x)\in W^1_2(0,\pi)$. 
According to (\ref{4}) and (\ref{5}) we have
\begin{equation}
F_{0_{tt}}(x,t)=\rho (t)F_{0_{xx}}(x,t),\
\ \rho (t)F_{xx}(x,t)=\rho (x)F_{tt}(x,t),
\label{14}
\end{equation}%
\begin{equation}
\left. F_{0}(x,t)\right\vert _{x=0}=0,\qquad\left.
F_{0}(x,t)\right\vert _{t=0}=0,  \label{15}
\end{equation}%
\begin{equation}
\frac{\partial }{\partial x}F_{0}(\mu ^{\pm }(x),t)=\pm \sqrt{\rho \left(
x\right) }\frac{\partial }{\partial \xi }\left. F_{0}\left( \xi ,t\right)
\right\vert _{\xi =\mu ^{\pm }(x)}.  \label{16}
\end{equation}%

Using the main equation (\ref{6}) it can be proved that 
\begin{equation}
A(x,0)=0,  \label{17}
\end{equation}%
\begin{equation}
\frac{\sqrt{\rho (x)}-1}{\sqrt{\rho (x)}+1}\frac{d}{dx}A(x,\mu ^{+}(x))=%
\frac{d}{dx}\left\{ A(x,\mu ^{-}(x)+0)-A(x,\mu ^{-}(x)-0\right\} .
\label{18}
\end{equation}

\subsection{Derivation of the Differential Equation}

\begin{lemma}
The following relations hold
\begin{equation}
-\varphi ^{\prime \prime }(x,\lambda )+q(x)\varphi (x,\lambda )=\lambda
^{2}\rho (x)\varphi (x,\lambda ),  \label{19}
\end{equation}%
\begin{equation}
\varphi (0,\lambda )=1,\text{ \ }\varphi ^{\prime }(0,\lambda )=0.  \label{20}
\end{equation}
\end{lemma}

\begin{proof}
Assume that $b(x)\in W_{2}^{2}(0,\pi )$ and
\[
J(x,\lambda ):=\frac{2}{1+\sqrt{\rho (t)}}A\left( x,\mu ^{+}(t)\right) +\frac{%
1-\sqrt{\rho (2a-t)}}{1+\sqrt{\rho (2a-t)}}A\left( x,2a-t\right) +
\] 
\begin{equation}
+F(x,t)+\int_{0}^{\mu ^{+}(x)}A(x,\xi )F_{0}(\xi ,t)d\xi =0, \label{21}  
\end{equation}
Differentiating (\ref{21}) twice with respect to $x$ and $t$ we get 
\[
J_{xx}^{\prime \prime }(x,t)-\rho (x)J_{tt}^{\prime \prime
}(x,t)-q(x)J(x,\lambda )\equiv 0.
\]%

Using the formulas (\ref{6}), (\ref{13})-(\ref{16}) and (\ref{18}), we obtain the following homogeneous equation
\[
\frac{2}{1+\sqrt{\rho (t)}}\left[A_{xx}\left( x,\mu
^{+}(t)\right)-\rho (x)A_{tt}\left( x,\mu
^{+}(t)\right)-q(x)A\left( x,\mu ^{+}(t)\right)\right]+
\]
\[+\frac{1-\sqrt{\rho (2a-t)}}{1+\sqrt{\rho (2a-t)}}\left[A_{xx}\left( x,2a-t\right)-\rho (x)A_{tt}\left( x,2a-t\right)-q(x)A\left( x,2a-t\right)\right]+
\]
\[+\int_{0}^{\mu ^{+}(x)}\left[ A_{xx}(x,\xi )-\rho (x)A_{\xi
\xi }(x,\xi )-q(x)A(x,\xi )\right] F_{0}(\xi ,t)d\xi =0.
\]
We know that from \cite{Mam-Kar2} this equation has only trivial solution:
\begin{equation}
A_{xx}(x,t)-\rho (x)A_{tt}
(x,t)-q(x)A(x,t)=0,\quad 0<t<x.  \label{23}
\end{equation}%

Differentiating (\ref{5}) twice, integrating by parts twice and using (\ref{17}) we obtain
\[
\varphi''(x,\lambda )+\lambda ^{2}\rho (x)\varphi (x,\lambda
)-q(x)\varphi (x,\lambda )=\varphi _{0}''(x,\lambda
)+\int_{0}^{\mu ^{+}(x)}A_{xx}(x,t)\cos \lambda tdt+
\]%
\[
-\lambda \rho (x)A(x,\mu ^{+}(x))\sin \lambda \mu ^{+}(x)+\sqrt{\rho (x)}%
A_{x}(x,\mu ^{+}(x))\cos \lambda \mu ^{+}(x)+
\]%
\[
+\lambda \rho (x)\sin \lambda \mu ^{-}(x)\left( A\left( x,\mu
^{-}(x)+0\right) -A\left( x,\mu ^{-}(x)-0\right) \right) +
\]%
\[
+\sqrt{\rho (x)}\cos \lambda \mu ^{-}(x)\frac{d}{dx}\left( A\left( x,\mu
^{-}(x)+0\right) -A\left( x,\mu ^{-}(x)-0\right) \right) +
\]%
\[
+\sqrt{\rho (x)}\cos \lambda \mu ^{+}(x)\left. \frac{\partial A(x,t)}{%
\partial x}\right\vert _{t=\mu ^{+}(x)}+
\]%
\[
+\sqrt{\rho (x)}\cos \lambda \mu ^{-}(x)\left( \left. \frac{\partial A(x,t)}{%
\partial x}\right\vert _{t=\mu ^{-}(x)+0}-\left. \frac{\partial A(x,t)}{%
\partial x}\right\vert _{t=\mu ^{-}(x)-0}\right) -
\]%
\[
-\varphi _{0}^{\prime \prime }(x,\lambda )+\lambda \rho (x)\sin \lambda \mu
^{+}(x)A(x,\mu ^{+}(x))+\rho (x)\cos \lambda \mu ^{+}(x)\left. \frac{%
\partial A(x,t)}{\partial t}\right\vert _{t=\mu ^{+}(x)}-
\]%
\[
-\lambda \rho (x)\sin \lambda \mu ^{-}(x)\left\{ A\left( x,\mu
^{-}(x)+0\right) -A\left( x,\mu ^{-}(x)-0\right) \right\} +
\]%
\[
+\rho (x)\cos \lambda \mu ^{-}(x)\left[ \left. \frac{\partial A(x,t)}{%
\partial t}\right\vert _{t=\mu ^{-}(x)-0}-\left. \frac{\partial A(x,t)}{%
\partial t}\right\vert _{t=\mu ^{-}(x)+0}\right] -
\]%
\[
-\rho (x)\int_{0}^{\mu ^{+}(x)}A_{tt}^{\prime \prime }(x,t)\cos \lambda tdt-
\]%
\[
-q(x)\left[ \frac{1}{2}\left( 1+\frac{1}{\sqrt{\rho (x)}}\right) \cos
\lambda \mu ^{+}(x)+\right. 
\]%
\[
+\left. \frac{1}{2}\left( 1-\frac{1}{\sqrt{\rho (x)}}\right) \cos \lambda
\mu ^{-}(x)+\int_{0}^{\mu ^{+}(x)}A(x,t)\cos \lambda tdt\right] .
\]%
Hence using (\ref{13}), (\ref{18}) and (\ref{23}) we arrive at (\ref{19}). The relations (\ref{20}) follow from (\ref{5}) for $x=0$. Lemma 2.1 is proved
in the case $b(x)\in W_{2}^{2}(0,\pi ).$

The proof of Lemma 2.1 in the case $b(x)\in W_{2}^{1}(0,\pi )$ is carried out
by a standard method (see e.g. \cite{Fre} p. 40).
\end{proof}
As in the theory of Sturm-Liouville problems (see \cite{Yur}, Lemma 1.5.8 and Corollary 1.5.1) the following lemmas can be proved.

\begin{lemma}
For each function $g(x)\in L_{2,\rho }(0,\pi )$,
\begin{equation}
\int_{0}^{\pi }\rho (x)g^{2}(x)dx=\sum_{n=1}^{\infty }\frac{1}{\alpha _{n}}%
\left( \int_{0}^{\pi }\rho (t)g(t)\varphi (t,\lambda _{n})dt\right) ^{2}.
\label{24}
\end{equation}
\end{lemma} 

\begin{cor}
For arbitrary functions $f(x),$ $g(x)\in L_{2,\rho }(0,\pi )$, 
\begin{equation}
\int_{0}^{\pi }\rho (x)f(x)g(x)dx=\sum_{n=1}^{\infty }\frac{1}{\alpha _{n}}%
\int_{0}^{\pi }\rho (t)f(t)\varphi (t,\lambda _{n})dt\int_{0}^{\pi }\rho
(t)g(t)\varphi (t,\lambda _{n})dt.  \label{25}
\end{equation}%
\end{cor}

Using the below lemmas the following lemma is proved with standard method.
\begin{lemma}
The following relation holds  
\begin{equation}
\int_{0}^{\pi }\rho (x)\varphi (t,\lambda _{n})\varphi (t,\lambda
_{k})dt=\left\{ 
\begin{array}{c}
0,\text{ \ \ \ }n\neq k \\ 
\alpha _{n},\text{ \ \ }n=k.
\end{array}%
\right.   \label{26}
\end{equation}%
\end{lemma} 

\subsection{Derivation of Boundary Condition}

\begin{lemma}
For all $n\geq 1$ the equality 
\[
\varphi (\pi ,\lambda _{n})=0
\]%
holds.
\end{lemma}

\begin{proof}
Since 
\[
-\varphi ^{\prime \prime }(x,\lambda _{n})+q(x)\varphi (x,\lambda
_{n})=\lambda _{n}^{2}\rho (x)\varphi (x,\lambda _{n}),
\]%
\[
-\varphi ^{\prime \prime }(x,\lambda _{m})+q(x)\varphi (x,\lambda
_{m})=\lambda _{m}^{2}\rho (x)\varphi (x,\lambda _{m}),
\]%
we get 
\[
\frac{d}{dx}\left( \varphi (x,\lambda _{n})\varphi ^{\prime }(x,\lambda
_{m})-\varphi ^{\prime }(x,\lambda _{n})\varphi (x,\lambda _{m})\right) =
\]%
\begin{equation}
=\left( \lambda _{n}^{2}-\lambda _{m}^{2}\right) \rho (x)\varphi (x,\lambda
_{n})\varphi (x,\lambda _{m})  \label{27}
\end{equation}%
From (\ref{27}) we have%
\[\left( \lambda _{n}^{2}-\lambda _{m}^{2}\right) \int_{0}^{\pi }\rho
(x)\varphi (x,\lambda _{n})\varphi (x,\lambda _{m})dx=
\]%
\[
=\varphi (\pi ,\lambda _{n})\varphi ^{\prime }(\pi ,\lambda _{m})-\varphi
^{\prime }(\pi ,\lambda _{n})\varphi (\pi ,\lambda _{m}).
\]%
By (\ref{26}) we get%
\begin{equation}
\varphi (\pi ,\lambda _{n})\varphi ^{\prime }(\pi ,\lambda _{m})-\varphi
^{\prime }(\pi ,\lambda _{n})\varphi (\pi ,\lambda _{m})=0.  \label{28}
\end{equation}%
Clearly, $\varphi ^{\prime }(\pi ,\lambda _{n})\neq 0,$ for all $n\geq
1.$ Indeed, if we suppose that $\varphi ^{\prime }(\pi ,\lambda _{m})=0$ for
a certain $m,$ then $\varphi (\pi ,\lambda _{m})\neq 0,$ and in view of (\ref{28}) $\varphi ^{\prime }(\pi ,\lambda _{n})=0$ for all $n.$

On the other hand,%
\[
\varphi'(\pi ,\lambda _{n})=\varphi _{0}'(\pi ,\lambda
_{n})+O( e^{|Im\lambda| \mu ^{+}(x)})
,\quad | \lambda| \rightarrow \infty 
\]%
i.e. for any $n,$ $\varphi ^{\prime }(\pi ,\lambda _{n})\approx \varphi
_{0}^{\prime }(\pi ,\lambda _{n}^{0})\neq 0$ as $n\rightarrow \infty ,$ that
contradicts the condition $\varphi ^{\prime }(\pi ,\lambda _{n})=0,$ $n\neq
m.$ Thus, $\varphi ^{\prime }(\pi ,\lambda _{n})\neq 0,$ for all $n\geq 1
$ and from (\ref{28}) \ we have 
\[
\frac{\varphi (\pi ,\lambda _{n})}{\varphi ^{\prime }(\pi ,\lambda _{n})}=%
\frac{\varphi (\pi ,\lambda _{m})}{\varphi ^{\prime }(\pi ,\lambda _{m})}=H,
\]%
i.e. for any $n$, $\varphi (\pi ,\lambda _{n})=H\varphi ^{\prime }(\pi
,\lambda _{n}).$ Since $\varphi (\pi ,\lambda _{n})=o(1)$ as $n\rightarrow
\infty ,$ we have $ H=0$ i.e. $\varphi (\pi ,\lambda _{n})=0.$
\end{proof}

Thus, we prove that the numbers $\left\{ \lambda _{n}^{2},\alpha _{n}\right\} _{n\geq
1}$ are spectral data of the constructed boundary value problem (\ref{1}), (\ref{2}). Then, the following theorem is proved.  
\begin{theorem}
For the sequences $\left\{ \lambda _{n}^{2},\alpha _{n}\right\} _{n\geq
1},$ where $\lambda _{n}\neq \lambda _{m}$ for $n\neq m,\alpha _{n}>0$ for
all $n$ to be spectral date of a problem $L(q(x))$ of the form (\ref{1})-(%
\ref{3}) with $q(x)\in L_{2}(0,\pi ),$ it is necessary and sufficient to
satisfy conditions 
\[
\lambda _{n}=\lambda _{n}^{0}+\frac{d_{n}}{\lambda _{n}^{0}}+\frac{k_{n}}{n}%
,\ \alpha _{n}=\alpha _{n}^{0}+\frac{t_{n}}{n},\ \left\{ k_{n}\right\}
,\left\{ t_{n}\right\} \in l_{2}
\]%
Here $\lambda _{n}^{0}$ are the zeros of the function 
\[
\Delta _{0}(\lambda )=\frac{1}{2}\left( 1+\frac{1}{\alpha }\right) \cos
\lambda \mu ^{+}(\pi )+\frac{1}{2}\left( 1-\frac{1}{\alpha }\right) \cos
\lambda \mu ^{-}(\pi ),
\]%
\[
\alpha _{n}^{0}=\int_{0}^{\pi }\varphi _{0}^{2}(x,\lambda _{n})\rho (x)dx,
\]%
\[
\varphi _{0}(x,\lambda )=\frac{1}{2}\left( 1+\frac{1}{\sqrt{\rho (x)}}%
\right) \cos \lambda \mu ^{+}(x)+\frac{1}{2}\left( 1-\frac{1}{\sqrt{\rho (x)}%
}\right) \cos \lambda \mu ^{-}(x),
\]%
\[
\mu ^{\pm }(x)=\pm x\sqrt{\rho (x)}+a\left( 1\mp \sqrt{\rho (x)}\right) ,
\]%
$d_{n}$ is a bounded sequence; $\left\{ k_{n}\right\} ,\left\{ t_{n}\right\}
\in l_{2.}$
\end{theorem}

Algorithm of the construction of the function $q(x)$ by spectral date $%
\left\{ \lambda _{n}^{2},\alpha _{n}\right\} $ follows from the proof of the
Theorem 2.1:

1) By the given numbers $\left\{ \lambda _{n}^{2},\alpha _{n}\right\}
_{n\geq 1}$ the functions $F_{0}(x,t)$ and $F(x,t)$ are constructed by
the formulas (\ref{7}) and (\ref{8}), respectively;

2) The function $A(x,t)$ is found from equation (\ref{6});

3) $q(x)$ is calculated by the formula (\ref{13}).

\section*{Acknowledgement} 
This work is supported by the Scientific and Technological Research Council of Turkey (TUBITAK).


\label{lastpage}
\end{document}